\documentclass[11pt,reqno]{amsart}

\usepackage{graphics,cancel,graphicx}
\usepackage{bbm,subfigure,rotating,bm,float,mathdots,wasysym}

\usepackage[all,cmtip]{xy}
\usepackage{amssymb,amsmath,amsfonts,amsthm,color}

\newcommand{\calD}{\mathcal{D}}

\newcommand{\mC}{\mathbb{C}}
\newcommand{\mD}{\mathbb{D}}

\newcommand{\mF}{\mathbb{F}}

\newcommand{\mN}{\mathbb{N}}

\newcommand{\mR}{\mathbb{R}}

\newcommand{\mT}{\mathbb{T}}

\newcommand{\mZ}{\mathbb{Z}}

\newtheorem{theorem}{Theorem}[section]
\newtheorem{lemma}[theorem]{Lemma}
\newtheorem{corollary}[theorem]{Corollary}
\newtheorem{proposition}[theorem]{Proposition}

\theoremstyle{definition}

\newtheorem{remark}[theorem]{Remark}

\theoremstyle{definition}
\newtheorem{definition}[theorem]{Definition}

\theoremstyle{definition}

\theoremstyle{definition}
\newtheorem{example}[theorem]{Example}
\newtheorem{observation}[theorem]{Observation}

\def\ord{\operatorname{ord}}

\def\T{ \mathbb T}

\def\H{H^\infty}
\def\D{{ \mathbb D}}
\def\C{{ \mathbb C}}

\def\N{{ \mathbb N}}

\def\inter{\cap}
\def\Inter{\bigcap }
\def\ov{\overline}

\def\ss{\subseteq}
\def\emp{\emptyset}

\def\buildrel#1_#2^#3{\mathrel{\mathop{\kern 0pt#1}\limits_{#2}^{#3}}}

\def\ssi{\Longleftrightarrow}

\begin{document}

\keywords{function algebras, Krull Intersection Theorem}

\subjclass[2010]{46E25 (primary), and 13A15, 13G05, 46J15 (secondary)}

\title[On the Krull Intersection Theorem]{On the Krull Intersection Theorem in Function Algebras}
\author[R. Mortini]{Raymond Mortini}
\address{
Universit\'{e} de Lorraine\\ D\'{e}partement de Math\'{e}matiques et  
Institut \'Elie Cartan de Lorraine,  UMR 7502\\
 Ile du Saulcy\\
 F-57045 Metz, France} 
 \email{Raymond.Mortini@univ-lorraine.fr}
\author[R. Rupp]{Rudolf Rupp}
\thanks{Part of this work was done while the 
second author enjoyed a sabbatical. 
He wants to thank TH N\"urnberg for it.}
\address{Fakult\"at f\"ur Angewandte Mathematik, Physik  und Allgemeinwissenschaften\\
\small TH-N\"urnberg\\
\small Kesslerplatz 12\\
\small D-90489 N\"urnberg, Germany}
\email{Rudolf.Rupp@th-nuernberg.de}
\author[A. Sasane]{Amol Sasane}
\address{Department of Mathematics \\London School of Economics\\
    Houghton Street\\ London WC2A 2AE\\ United Kingdom}
\email{sasane@lse.ac.uk}

\begin{abstract} 
A version of the Krull Intersection Theorem states that for 
Noetherian domains   the 
{\em Krull intersection} $ki(I)$ of  every proper ideal $I$ is trivial; that is 
$$
ki(I):=\displaystyle\bigcap_{n=1}^\infty I^n = \{0\}.
$$
We investigate the validity of this result for various function
algebras $R$, present  ideals $I$ of $R$ for which $ ki(I)\neq
\{0\}$, and give conditions on $I$ so that $ki(I)=\{0\}$.
\end{abstract}

\maketitle

\section{Introduction}

\noindent The aim of this note is to investigate the validity of the
Krull Intersection Theorem in various function algebras.  We begin by
recalling the following version of the Krull Intersection Theorem
\cite[Corollary~5.4, p.152]{E}. See also \cite{Per} for a simple proof. 
As usual, given an ideal $I$, $I^n$ is the ideal of all elements  of the form
$$
\sum_{i=1}^m a_{1,i}\cdots a_{n,i}, \quad m\in \N, \;a_{k,i}\in I.
$$

\begin{proposition}[Krull Intersection Theorem] 
\label{propo_KIT}
If $R$ is a Noetherian integral domain, and $I$  a proper ideal 
 of $R$, that is $I\subsetneq R$, then the {\em Krull intersection} 
$ki(I)$ of $I$, defined by
$$
ki(I):=\displaystyle\bigcap_{n=1}^\infty I^n ,
$$
is trivial, that is, $ki(I)=\{0\}$.
\end{proposition}

\noindent We note that neither of the assumptions on $R$ can be
dropped. Here are some examples.

\begin{example}[Not Noetherian, and not an integral domain] 
This is based on \cite[p.153]{E}.  Let $R=C^\infty(\mR)$, 
the ring of all infinitely differentiable real-valued functions on
$\mR$.  Then $R$ is not Noetherian (since
$$
I_n:=\{f\in C^\infty(\mR): f(x)=0 \text{ for }x>n\}, \quad 
n\in \mN:=\{1,2,3,\cdots\},
$$
form an ascending chain of ideals) and is not an integral domain. Let
$I$ be the ideal $\langle x\rangle$ generated by $x\mapsto x$.  Let
$$
f(x):=\left\{\begin{array}{ll}
e^{-1/x} & \textrm{if }x>0,\\
0 &\textrm{otherwise}.
\end{array}\right.
$$
Then $f\in C^\infty(\mR)$. For $n\in \mN$, set
$$
f_n(x)=\left\{\begin{array}{ll}
f(x)/x^n & \textrm{if }x>0,\\
0 &\textrm{otherwise}.
\end{array} \right.
$$
Then $f_n\in C^\infty(\mR)$ too, and so $f=f_n x^n\in I^n$. So we have
$0\neq f\in ki(I)$.  \hfill$\Diamond$
\end{example}

\begin{example}[Not Noetherian, but an integral domain]
\label{example_entire}
Let 
$$
R=H(\mC)=\{f:\mC\rightarrow \mC:f\textrm{ is entire}\}.
$$
Denote the zero-set of  a function $f\in H(\mC)$ by 
$$Z(f)=\{z\in \mC: f(z)=0\}.$$
If $z_0\in Z(f)$, let $\ord(z_0,f)$ be the order of $z_0$ as a zero of $f$.  
Define
$$
I:=\{0\}\cup\left\{f\in H(\mC)\Big|\; \begin{array}{ll} 
\exists N\in \mN, \forall n\in \mN: {\rm if}\; n>N,\;  {\rm then}  \;f(n)=0 \\ {\rm and}\; 
\displaystyle \lim_{n\rightarrow \infty} \ord(n,f)=\infty 
\end{array}\right\}.
$$
Then it can be seen that $I$ is an ideal.  We will show that
$I^2=I$. To this end, let $0\neq f\in I$. Let $f_1\in H(\mC)$  be
an entire function  with $Z(f_1)=Z(f)\inter \N$, but such that for each $n\in Z(f_1)$,
$$
\ord(n,f_1):=\max\left\{ 1,\Big\lfloor \frac{\ord(n,f)}{2}\Big\rfloor \right\}.
$$
Here for $x\in \mR$, the notation $\lfloor x\rfloor$ stands for the
largest integer $\leq x$.  Then $f_1\in I$. Set $f_2=f/f_1$. Then
$f_2\in I$ as well. Finally, $f=f_1 \cdot f_2\in I\cdot I=I^2$.
\hfill$\Diamond$
\end{example}

\begin{example}[Noetherian, but not an integral domain] 
The ring $\mC[z]$ is Noetherian, and so it follows that
$R:=\mC[z]/\langle z^2 -z\rangle$ is Noetherian too. But $R$ is not an
integral domain because
 $
[z][z-1]=[0].
$ 
With $I:=\langle [z]\rangle $, the ideal generated by $[z]$ in $R$, we
see that $I^2=I$,   because $[z]$ is an idempotent, and so 
$I=\langle [z]\rangle =\langle [z]^2\rangle  \ss I^2 \ss I$. \hfill$\Diamond$
\end{example}

\begin{example}[Sequence spaces] 
Consider the sequence algebras 
\begin{eqnarray*}
\ell^2 &:=&\Big\{(a_n)_{n\in \mZ}\;\Big|\; \sum_{n\in\mZ} |a_n|^2 < \infty\Big\},\\
\ell^\infty &:=& \Big\{ (a_n)_{n\in \mZ}\;\Big|\;  \sup_{n\in \mZ}|a_n|<\infty\Big\},\\
s'(\mZ)&:=& \Big\{ (a_n)_{n\in \mZ}\;\Big|\;  \exists M>0,\;\exists m>0, 
\;\forall n\in \mZ: \; |a_n|\leq M(1+|n|)^m\Big\}, 
\end{eqnarray*}
endowed with termwise addition, termwise scalar multiplication, and
termwise (Hadamard) multiplication.  Then for any of the above
algebras $R$, $I:=c_{00}$, the set of all sequences with compact
support, is a proper ideal in $R$. If $a:=(a_1,\cdots,
a_N,0,\cdots)\in c_{00}$, then with $b_n$ any complex number such that
$b_n^2=a_n$, $n=1,\cdots ,N$, and with
 $ 
b:=(b_1,\cdots, b_N,0,\cdots)\in c_{00}=I,
$ 
we have that $I\owns a=b\cdot b\in I^2$. So $I^2=I$, and hence
$ki(I)=I\neq \{0\}$. We remark that $\ell^2(\mZ)$ with the termwise
operations is isomorphic to $L^2(\mT)$ with convolution, where
$\mT:=\{z\in \mC:|z|=1\}$, and $s'(\mZ)$ with termwise operations is
isomorphic to the algebra of periodic distributions $\calD'(\mT)$ with
convolution.  \hfill$\Diamond$
\end{example}

\noindent Here is an example of a non-Noetherian ring for which $\bigcap_{n=1}^\infty I^n=\{0\}$
for every proper ideal $I$. (This example is included in \cite[Theorem~4]{And}, 
but we offer an elementary direct proof below.) 

\begin{example}[Non-Noetherian, but $\displaystyle \bigcap_{n=1}^\infty I^n=\{0\}$ 
for all ideals $I\subsetneq  R$] 

Let $I$ be a proper ideal in $R:=\mC[z_1, z_2, z_3,\cdots]$.
Then $I$ is contained in some maximal ideal $M$ of $R$.
But then $ki(I) \subseteq ki(M)$. We will show that the maximal ideals $M$ 
of $R$ are just of the form $\langle z_n - \zeta_n:  n \in \mN\rangle$ for some 
sequence of complex numbers $\zeta_1,\zeta_2,\zeta_3,\cdots$.
 Then  we will use this structure to show $ki(M)=\{0\}$, and hence we can conclude that 
 also $ki(I)=\{0\}$ for every proper ideal $I$ in $R$. 
 
 \medskip

\noindent {\bf Claim:} $M$ is a maximal ideal in 
$ \mC[z_1,z_2,z_3,\cdots]$ if and only if there exists a sequence 
$\zeta_1, \zeta_2,\zeta_3,\cdots$ of complex numbers such that 
$$
M=\langle z_n-\zeta_n: n\in \mN\rangle.
$$
(This result is known; see \cite{Lan}. Nevertheless, we include an elementary self-contained 
proof, fashioned along the same lines as the proof of Hilbert's Nullstellensatz; see 
\cite{Ver}.) 

\medskip 

\noindent If $\bm{\zeta}=(\zeta_n)_{n\in \mN}$ is a sequence of complex numbers, then 
 we first observe that the ideal $M_{\bm{\zeta}}:=\langle z_n-\zeta_n: n\in \mN\rangle$ is 
maximal as follows. We can look at the evaluation homomorphism $\varphi_{\bm{\zeta}}$ 
 from $R:=\C[z_1,z_2,z_3,\cdots]$ to $\C$ sending 
 for every $n\in\N$ the indeterminate $z_n$ to $\zeta_n$. 
 Then $\varphi_{\bm{\zeta}}$  is surjective, and  
  $
 \ker \varphi_{\bm{\zeta}}
 = \{p \in R: p(\bm{\zeta})=0\}.
 $ 
 But as $p(\bm{\zeta})=0$, it follows from 
 the Taylor series centered at $\bm{\zeta}$ for  $p$ that $p$ belongs to $M_{\bm{\zeta}}$. 
 Hence by the isomorphism theorem $R/M_{\bm{\zeta}}$ is isomorphic to $\mC$, and 
 thus $M_{\bm{\zeta}}$ is maximal.

\medskip 

\noindent 
Now suppose that $M$ is maximal. 
Let $k\in \mN$, and consider the ring homomorphism  
$$
\begin{array}{rcl}
\varphi_k : \mC[z_k]&\longrightarrow& \mC[z_1,z_2,z_3,\cdots]/M=:\mF\\
p& \longmapsto & [p]:=p+M.
\end{array}
$$
If $\varphi_k(pq)=0$, then since $\mF$ is a field, either $\varphi_k(p)=0$ or $\varphi_k(q)=0$. 
Hence $\ker \varphi_k$ is a prime ideal of $\mC[x_k]$. We will show first that $\ker \varphi_k\neq\{0\}$. 

Suppose that $\ker \varphi_k=\{0\}$. Then $\varphi_k:\mC[z_k]\rightarrow \mF$ 
is an injective map. Thus there exists an extension of $\varphi_k$ to $\Phi_k:\mC(z_k)\rightarrow \mF$, 
namely 
$$
\Phi_k\Big(\Big[\frac{p}{q}\Big]\Big)=\frac{\varphi_k(p)}{\varphi_k(q)}, \quad 
\Big[\frac{p}{q}\Big]\in \mC(z_k), \;p\in \mC[z_k],\; q\in \mC[z_k]\setminus \{0\}.
$$
It is straightforward to check that $\Phi_k$ is an injective homomorphism.
Now
$\mF$ is a $\mC$-vector space which is spanned by a countable number of elements: 
namely $m+M$, where $m$ is any monomial in $\mC[z_1,z_2,z_3,\cdots]$.
Hence its 
subspace, namely $\Phi_k(\mC(z_k))$ is also spanned by a countable sequence of vectors, say 
$\{v_n:n\in \mN\}$. As these $v_n$ are in $\Phi_k(\mC(z_k))$, there exist $r_n\in \mC(z_k)$ 
such that $\Phi_k(r_n)=v_n$. But then 
$\mC(z_k)$ will be spanned by the $r_n$: indeed, if $r\in \mC(z_k)$, then $\Phi_k(r)\in \mC(z_k)$ 
and so there exist $\alpha_1,\cdots,\alpha_m\in \mC$ such that 
 $
\Phi_k(r)= \alpha_1 v_{\ell_1}+\cdots+ \alpha_m v_{\ell_m}
=\Phi_k(  \alpha_1 r_{\ell_1}+\cdots+ \alpha_m r_{\ell_m}),
$ 
and so $r=\alpha_1 r_{\ell_1}+\cdots+ \alpha_m r_{\ell_m}$, thanks to the injectivity of $\Phi_k$. 
 So  the $\mC$-vector space $\mC(z_k)$ is also spanned by a countable number of vectors. 
However, it is easy to see that 
$$
\left\{z_k\mapsto \frac{1}{z_k-\zeta}: \zeta\in \mC\right\},
$$
is an uncountable linear independent set
 in $\mC(z_k)$, a contradiction. Hence $\ker \varphi_k\neq \{0\}$. 
The kernel of $\varphi_k$ is proper, because $\varphi_k(1)=[1]$
(note that  due to $\varphi_k(1)=\varphi_k(1^2)$, the only other possibility would be  $\varphi_k(1)=[0]$, giving $1+M =[0]$;  a contradiction).
Since $\ker\varphi_k$ is a nonzero proper prime ideal in $\mC[x_k]$, it follows that there is a $\zeta_k\in \mC$ 
such that $z_k-\zeta_k\in \ker \varphi_k$. Hence $z_k-\zeta_k\in M$. 

As the choice of $k\in \mN$ was arbitrary, we get a sequence $
\zeta_1, \zeta_2,\zeta_3,\cdots$ of complex numbers such that 
 $
z_n-\zeta_n\in M, 
$ 
and so $\langle z_n-\zeta_n: n\in \mN\rangle\subseteq M$. But $\langle z_n-\zeta_n: n\in \mN\rangle$ is 
maximal. Thus  $M=\langle z_n-\zeta_n: n\in \mN\rangle$. 
This completes the proof of the claim.

\medskip 

\noindent 
Let $M=\langle z_n-\zeta_n:  n \in \mN\rangle$. 
 Suppose that $0 \neq f \in ki(M)$, and let $k$ be such that $f\in  \mC[z_1,\cdots ,z_k]$. 
Let $\Pi_k:\mC[z_1,z_2,z_3,\cdots]\rightarrow \mC[z_1,\cdots, z_k]$ be the evaluation homomorphism that sends $z_n$ to $\zeta_n$ for $n>k$, and 
$z_n$ to $z_n$ itself if $n \leq k$. Then 
$$
f \in \bigcap_{n=1}^\infty \big(\Pi_k(M)\big)^n.
$$
 Since $\Pi_k(M)$ is just the maximal ideal $\langle z_n-\zeta_n: 0 \leq n\leq k\rangle $
in the ring $\mC[z_1,\cdots,z_k]$,  we conclude from  the Krull Intersection Theorem (Proposition~\ref{propo_KIT}) 
applied to the Noetherian integral domain $\mC[z_1,\cdots,z_k]$, that  $f=0$. Consequently,  $ki(M)=\{0\}$. 
\hfill$\Diamond$
\end{example}

\noindent In this rest of this article, we will investigate $ki(I)$ for (mainly maximal) ideals  $I$ in 
 algebras of (mainly holomorphic) functions. The organization of the subsequent sections is as follows:
\begin{enumerate}
  \item In Section~\ref{section_1}, we will determine $ki(I)$ for certain ideals in the 
  algebra $H(D)$ of holomorphic functions in a domain $D\subseteq\mC$. 
  \item In Section~\ref{section_3}, we will determine $ki(I)$ for certain ideals in uniform algebras. 
  \item In Section~\ref{section_4}, we will determine $ki(I)$ for certain ideals in the algebra 
  $\H(\D)$ of bounded holomorphic functions in the unit disk $\D$.
\end{enumerate}

\section{$ki(I)$ for ideals $I$ in $H(D)$}
\label{section_1}

\noindent Example~\ref{example_entire} above can be generalized to the
following.

\begin{example}
Let $D\subseteq\mC$ be a domain (that is, an open path-connected set). 
Let $R=H(D)$, the algebra of holomorphic functions in $D$ with
pointwise operations. Then there exists a proper ideal $I$ of $R$ such
that $ki(I)\neq \{0\}$. We construct such an ideal $I$ as follows.
Let $(\zeta_n)_{n\in \mN}$ be any sequence in $D$ that converges to a
point in the boundary $\partial D$ of $D$ (or more generally, without
accumulation points in $D$). Let $h$ be any Weierstrass product with
simple zeros at $\zeta_n$, $n\in \mN$. Consider the proper ideal $I$
of $R$ generated by the functions $f_n$, $n\in \mN$, given by
$$
f_n(z):=\frac{h(z)}{(z-\zeta_1)\cdots (z-\zeta_n)}, \quad z\in D.
$$
Let $g\in R\setminus \{0\}$ be a Weierstrass product which vanishes
exactly at $\zeta_n$ of order $\ord(g,\zeta_n)=n$, for each $n\in \mN$.
We claim that $g\in I^n$ for $n\in \mN$. For $n\in \mN$, set
$$
q_n:=\prod_{k=1}^n (z-\zeta_k)^k.
$$
Then $\displaystyle G:=g/q_n $ has the zero set
 $
Z_D(G)=\{\zeta_{n+1}, \zeta_{n+2},\cdots \},
$  
with orders of zeros $\zeta_k$ given by $\ord(G,\zeta_k)=k$, for $k> n$.
Again by the Weierstrass' Factorization Theorem, we must have that
$f_n^n$ divides $G $ in $H(D)$, and hence there exist $h_n\in H(D)$
such that $g=q_n G=q_n h_n f_n^n\in I^n$.  Since $n\in \mN$  was 
arbitrary, $g \in ki(I)$. \hfill$\Diamond$
\end{example}

\noindent On the other hand, we have the following result saying that
for {\em non-free}/{\em fixed} ideals $I$ of $H(D)$, $ki(I)=\{0\}$.
 
\begin{definition}[Free ideals in $H(D)$] 
Let $D\subseteq\mC$ be a domain.  For an element  $f\in H(D)$, 
let $Z_D(f)$ denote the set of zeros of $f$.  An ideal $I$ in $H(D)$
is called {\em free} if the {\em zero set} of $I$,
$$
Z_D(I):=\bigcap_{f\in I} Z_D(f)
$$
is the empty set $\emptyset$, and {\em fixed/non-free} otherwise.
\end{definition}
 
\begin{proposition}
For a proper fixed ideal $I$ of $H(D)$, $ki(I)=\{0\}$.
\end{proposition}
\begin{proof} 
Suppose that $\zeta\in Z_D(I)$. Let 
$$m:= \min\{\ord(f,\zeta): f\in I\}.
$$
Then $m\geq 1$
 and each function $g\in I^n$ has $\zeta $ as a
zero of order at least $ mn$. But any holomorphic function belonging
to $I^n$ for all $n\in \mN$ must therefore be identically zero since
$D$ is a domain.
\end{proof}

\noindent For maximal ideals of $H(D)$, one can say more, and we have
the following results given in Theorem~\ref{thm_MR}. But first we
proof a helpful lemma, which will be used in the proof of
Theorem~\ref{thm_MR} (and also in the subsequent result).
 
\begin{lemma}
\label{little_lemma}
Let $D$ be a domain, and $M$ be a maximal ideal in the ring $H(D)$.  If $f \in M$, 
and $h\in H(D)$ is such that $Z_D(f) \ss Z_D(h)$ $($disregarding
multiplicities$)$, then $h \in M$ too.
\end{lemma}
\begin{proof} 
Suppose that $h$ is not in $M$. Then the ideal $\langle h\rangle +M$, 
which strictly contains $M$, must be $H(D)$, thanks to the maximality
of $M$. Thus there exists $m\in M$ and $g\in H(D)$ such that $1=gh+m$.
Hence we have that $Z_D(h) \cap Z_D(m)=\emptyset$. Thus $Z_D(f) \cap
Z_D(m)=\emptyset$.  By the Nullstellensatz for $H(D)$, it follows that
there exist $u,v\in H(D)$ such that $1 =uf +vm\in M$.
This is absurd,  because $m,f\in M$ and  $M$ is proper. 
\end{proof}
 
\begin{theorem}
\label{thm_MR}
Let $D$ be a domain in $\mC$, and $M$ be a free maximal
ideal of $H(D)$. Then
\begin{enumerate}
\item $\{0\}\cup \Big\{f\in M: f\neq 0 \textrm{ and
  }\displaystyle \lim_{\zeta\in Z_D(f)} \ord(f,\zeta)=\infty\Big\} \subset
  ki(M).$
\item $ki(M)\subseteq\{0\}\cup \Big\{f\in M:
  f\neq 0 \textrm{ and }\displaystyle \sup_{\zeta\in Z_D(f)}
  \ord(f,\zeta)=\infty\Big\}$.
\end{enumerate}
Hence $ki (M)=\{0\}$ if and only if there exists a $\zeta
\in D$ such that $M=\langle z-\zeta\rangle$.
\end{theorem}

\noindent Here, by assumption that
$$
\displaystyle \lim_{\zeta\in Z_D(f)} \ord(f,\zeta)=\infty,
$$
we mean that given any $n>0$, there exists a finite set $K\subset
Z_D(f)$ such that $\ord(f,\zeta)>n$ for all $\zeta \in Z_D(f)\setminus
K$.

\begin{proof} 
(1)  First we observe that $M$ contains no polynomial. 
(Otherwise, if a polynomial $p\in M$, it follows, by using the fact
that $M$ is in particular prime, that $M$ contains a linear factor
$z-w$ of $p$.  But then we have that $M \ss M_w:=\{f \in
H(D):f(w)=0\}$.  Since the later ideal is proper, and $M$ is maximal,
$M =M_w$ would be a fixed ideal.)
  
Let $f\in M\setminus \{0\}$ with
$$
\lim_{\zeta\in Z_D(f)} \ord(f,\zeta)=\infty.
$$ 
Suppose that $n\in \mN$. Then it is possible to factorize $f$ as
$f=f_n p$, where $p$ is a polynomial and the orders of all zeros of
$f_n$ are at least $n$.  By the primeness of (the maximal!) ideal $M$,
and the fact that $M$ contains no polynomials, it follows that $f_n
\in M$ too. But now we can write $f_n=g_1\cdots g_n$, where each of
the functions $g_k$ have the same zero set (disregarding
multiplicities).  Again the primeness of $M$, and
Lemma~\ref{little_lemma}, allow us to conclude that all the $g_k$
belong to $M$. Hence $f\in M^n$. As the choice of $n\in \mN$ was
arbitrary, we obtain that $f\in ki(M)$.
   
\medskip 

\noindent (2) Assume that $f\in ki(M)$ and let  $n\in \mN$. Then $f$ can be
decomposed into a finite sum of the form
$$
f=\sum_{k=1}^N  f_{1,k}\cdots f_{n,k},
$$
with each $f_{j,k}\in M$. All  these functions $f_{j,k}\in M$ must
have a common zero, since otherwise (by the Nullstellensatz for
$H(D)$), we can generate $1$ in $M$, a contradiction to the fact that $M$ is proper.
 But then the order of this common zero of $f$ must be at least
$n$. As the choice of $n\in \mN$ was arbitrary, it follows that
$\displaystyle \sup_{\zeta\in Z_D(f)} \ord(f,\zeta)=\infty.$
\end{proof}

\noindent We remark that a somewhat different characterization of
$ki(M)$ was provided in \cite[Theorem~3,p.714]{Hen} for the algebra
$H(\mC)$ of entire functions. We extend Henriksen's result to domains,
and then compare our result above with his result below. 
Since Henriksen's proof was, in our viewpoint,  very condensed, we provide all details in the more general case. Since prime ideals
appear very naturally in the description of $ki(M)$, we include a nice property shared by
 this class of ideals. Also that result is known; \cite[Theorem~1]{Hen}.

\medskip 
 
\noindent Given $f\in H(D)$, let 
$$
\mathbf{o}(f):=\sup_{\zeta \in Z_D(f)} \ord(f,\zeta).
$$
If $q\equiv 0$, we set $\mathbf{o}(q):=\infty $.

\begin{lemma}\label{prime-o}
Let $P$ be  a prime ideal in $H(D)$, where $D$ is a  domain in $\C$.
Then $P$ is non-maximal if and only if $\mathbf o(f)=\infty$ for every $f\in P$.
\end{lemma}
\begin{proof}
The only if direction is a direct consequence of Lemma \ref{little_lemma}.
So suppose that $P$ is prime and  contains an element $f$ with $N:=\mathbf o(f)<\infty$.
Let $M$ be  a maximal ideal with $P\ss M$. We show that $P=M$.
Write $f=f_1\dots f_N$, where each zero of any $f_j$ is simple. Since $P$ is prime,
at least one of the $N$ factors belongs to $P$. Say it is $f_1$. Fix $g\in M $ and let
$d=\gcd(f_1,g)$. Then, by Wedderburn's Theorem \cite[p.119]{Rem},  $d\in \langle f_1, g\rangle\ss M$.
Now $f_1=dh $ for some $h\in H(D)$.  Since the zeros of $f_1$ are simple,
$Z(d)\inter Z(h)=\emp$. Hence, $h$ cannot belong to $P\ss M$, because otherwise
$H(D)=\langle d,h\rangle \ss M$, a contradiction. Thus $d\in P$ and so  $g\in P$. Consequently,
$M=P$.
\end{proof}

\begin{proposition}[Henriksen] 
\label{prop_Henriksen}
Let $D\ss \mC$ be a domain, and let $M$ be a free, maximal ideal
of $H(D)$.  Then
$$
ki(M) 
= 
\left\{\!f\in
M\bigg|
 \begin{array}{ll} \textrm{\em whenever } d\in H(D)\setminus M 
\textrm{\em is a divisor of}\;  f, \\ \textrm{\em say}\;  f=q\cdot d,
\textrm{\em we have }
\;\displaystyle \mathbf{o}(q)=\infty
\end{array} \!\right\}.
$$
Moreover
$ki(M)$ is the largest nonmaximal prime ideal contained in $M$.
\end{proposition}
\begin{proof}
Since every finitely generated ideal of $H(D)$ is principal, 
$ki(M)$ is easily seen to be the set of all $f\in H(D)$ such that for
all $n \in \mN$, we have a factorization $f=h_nd_n^n$, with $h_n\in
H(D)$, $d_n\in M$.

\medskip 

Let 
$$
K:= 
\left\{\!f\in
M\bigg|
 \begin{array}{ll} \textrm{whenever } d\in H(D)\setminus M 
\textrm{ is a divisor of}\;  f, \\ \textrm{say}\;  f=q\cdot d,
\textrm{we have }
\;\displaystyle \mathbf{o}(q)=\infty
\end{array} \!\right\}.
$$
We first prove that $ki(M)\subseteq K$. Let $f \in ki(M)$ and suppose
that $d$ is a divisor of $f$ which does not belong to $M$. Say $f=d
q$. We need to show that $\mathbf{o}(q)=\infty$. If not, then let
$n:=\mathbf{o}(q)$ and $k=2n$.  Since $f\in ki(M)$, there exists an
$h_k\in H(D)$ and a $g_k\in M$ such that $dq=f=h_kg_k^k$. But then
every zero of $g_k$ must be a zero of $d$ (disregarding multiplicities)
(because each zero of $q$ appears at most $n$ times; on the other hand
every zero of $g_k$ appears at least $2n$ times).  Thus $Z_D(g_k)\ss
Z_D(d)$. Since $g_k\in M$, we have $d\in M$ by our
Lemma~\ref{little_lemma}, a contradiction.
   
Next we will show that $K\subseteq ki(M)$. Given $f\in K\subseteq M$, and
 $n\in \mN$, we may factor $f\in M$ as $f=f_1 f_2$, 
where $Z(f_2)=\{\zeta\in Z(f): \ord(f,\zeta) \geq n+1\}$
and $Z(f_1)=\{\zeta\in Z(f): \ord(f,\zeta) \leq n\}$.
If one of these sets is empty, we just let the associated function equal to be $1$.

 If $f_2\not\in M$, then we end up with $f_1\in M$.
But the definition of $K$ now implies that $\infty=\mathbf{o}(f_1)\leq
n$.  Thus in our factorization $M\owns f=f_1f_2$, we have $f_2\in
M$. Take a function $h_n$ such that we have $Z_D(h_n)=Z_D(f_2)$, and
such that $h_n$ has only simple zeros. Then by
Lemma~\ref{little_lemma}, $h_n\in M$ because $f_2\in M$. By
construction, $h_n^n$ divides $f_2$, and so $f_2=gh_n^n$. Summarizing,
$f=f_1f_2=f_1gh_n^n\in M^n$. Since $n\in \mN$ was arbitrary, it
follows that $f\in ki(M)$. This completes the proof that $ki(M)=K$.
   
\medskip 
  
Next we show that $ki(M)$ is prime. Assume that $f=f_1\cdot f_2\in ki(M)
\subseteq M$. Since $M$ is prime, we have one of three possible cases:
\begin{itemize}
 \item[$1^\circ$] $f_1\in M$ and $f_2\not\in M$, 
 \item[$2^\circ$] $f_1\not\in M$ and $f_2\in M$, 
 \item[$3^\circ$] $f_1\in M$ and $f_2\in M$. 
\end{itemize}

\noindent Case $1^\circ$: Let $d\in H(D)\setminus M$ be a divisor of $f_1$. Say $f_1=gd$. Then $f=g(df_2)$, where 
$df_2\notin M$. Since $ki(M)=K$, we deduce that $g\in ki(M)$. So $f_1\in ki(M)$.
Case $2^\circ$ works in the same way.

Now only the case left is when both $f_1,f_2$ are in $ M $.  Assuming
that neither $f_1$ nor $f_2$ belongs to $ki( M )$, we proceed as
follows.  In this case, there exist $d_i$ dividing $f_i$, with $d_i\in
H(D)\setminus M $, $q_i:=f_i/d_i\in M $, and $\mathbf{o}(q_i)<\infty$,
$i=1,2$. Since $ M $ is maximal, and in particular prime, $d_1 d_2\in
H(D)\setminus M $, and $ \mathbf{o}(q_1 q_2)\leq \mathbf{o}(q_1)+
\mathbf{o}(q_2) <\infty$. So $f_1f_2\not\in ki( M )$.

Consequently, $ki( M )$ is prime. Finally, we will show the following:

\medskip 

\noindent {\bf Claim:} $ki(M)$ is the largest nonmaximal prime ideal
contained in $M$.

\noindent First we show that $ki(M)$ is not maximal.
Take any nonzero $f\in M$, and let $h\in H(\mD)$ be such
that $Z(h)=Z(f)$, but $\ord(h,\zeta)=1$ for all $\zeta\in
Z_D(h)$. Then by Lemma~\ref{little_lemma}, $h\in M$ too. But with
$d:=1\in H(D)\setminus M$, and $q:=h$, we have $f=qd=h\in M$, but
$\mathbf{o}(q)=1<\infty$. Hence $f=h\not\in ki(M)$. Thus
$ki(M)\subsetneq M$, and so $ki(M)$ is nonmaximal.

Suppose now that $P$ is a 
prime ideal such that $ki(M)\subsetneq P\subseteq M$. 
Let $f\in P\setminus ki(M)$. Then there exists  $d\in
H(D)\setminus M$ and  $q\in M$ such that $f=q\cdot d$ and
$\mathbf{o}(q)<\infty$. But as $d\not\in M$ and hence not in $P$
either, we have $q\in P$.  By Lemma \ref{prime-o}, $P=M$.

This completes the proof of Proposition~\ref{prop_Henriksen}. 
\end{proof}

\begin{example} 
The aim of this example is to contrast the results 
from Theorem~\ref{thm_MR} and Proposition~\ref{prop_Henriksen}. If we
call
\begin{eqnarray}
\label{set_A}
  A&:=&\{0\}\cup \Big\{f\in  M : f\neq 0 \textrm{ and
  }\displaystyle \lim_{\zeta\in Z_D(f)} \ord(f,\zeta)=\infty\Big\} ,\\
\label{set_B}
  B&:=&\{0\}\cup \Big\{f\in  M : f\neq 0 \textrm{ and
  }\displaystyle \sup_{\zeta\in Z_D(f)} \ord(f,\zeta)=\infty\Big\} ,
\end{eqnarray}
then in Theorem~\ref{thm_MR} we have shown that 
$A\subseteq ki(M)\subseteq B$ whenever $M$ is a maximal free ideal in $H(D)$.
  We will show that
\begin{enumerate}
\item there exists an element $f\in B\setminus ki( M )$, showing that
  $B\neq ki( M )$;
\item there exists an element $g\in ki( M )\setminus A$, showing that
  $A\neq ki( M )$.
\end{enumerate} 

\medskip

 \noindent 
 To this end, first note that $A$ and $B$ are not ideals. 
 In fact, concerning $A$, just consider $f\in A$ and multiply $f$ 
by a function with simple zeros outside $Z(f)$.  Concerning $B$, 
let $f\in M$ have simple zeros (for the existence, see 
Lemma  \ref{little_lemma}). Now let $g_1,g_2$ be in $H(D)$ with 
$Z(g_1)\inter Z(g_2)=\emp$ and $\mathbf o(g_j)=\infty$.
 Choose $a_j\in H(D)$ so that $1=a_1g+a_2g_2$. Then $fa_jg_j\in B$, but 
 $$fa_1g_1+fa_2g_2=f\notin B.$$
Hence we conclude that $A\subset ki(M)\subset  B$, the inclusions being strict. 
\end{example}


\section{Sufficient conditions for $ki(I)=I$ in uniform algebras}
\label{section_3}

\noindent We recall the definition of a uniform algebra.

\begin{definition}[Uniform algebra] 
$R$ is called a {\em uniform algebra on $X$} if 
\begin{enumerate}
\item $X$ is a compact topological space,
\item $R\subseteq C(X;\mC)$, the algebra of complex-valued continuous
  functions on $X$, and $R$ separates the points of $X$, that is, for
  every $x,y\in X$ with $x\neq y$, there exists $f\in R$ such that
  $f(x)\neq f(y)$,
\item the constant function $1\in R$,
\item $R$ is a closed subalgebra of $C(X;\mC)$, where the latter is
  endowed with the usual supremum norm $\|\cdot\|_\infty $.
\end{enumerate}
\end{definition}

\noindent We also recall below the following two well-known results
from the theory of uniform algebras; see \cite[Lemma~1.6.3, p.72-73
and Theorem~1.6.5, p.74]{B}. Both of these results involve the notion
of an approximate identity, given below.

\begin{definition}[Approximate identity] 
Let $R$ be a commutative unital Banach algebra, and $ M $ 
be a maximal ideal of $R$. We say that $ M $ has an 
{\em approximate identity} if there exists a constant $K$ such that
for every $\epsilon>0$, and every $f_1,\cdots, f_n\in M $, there
exists an $e\in M $, $\|e\|\leq K$, such that $\|ef_i-f_i\|<\epsilon$
for all $i=1,\cdots , n$. (In other words, there exists a bounded net
$(e_\alpha)$ in $ M $ such that $e_\alpha f\rightarrow f$ for every
$f\in M $.)
\end{definition}

\begin{proposition}
Let $R$ be a uniform algebra on $X$, and let $x\in X$. Then the 
following are equivalent:
\begin{enumerate}
\item $($Existence of an approximate identity.$)$ The maximal ideal
$$
 M :=\{f\in R:f(x)=0\}
$$ 
has an approximate identity.
\item $($Existence of a weak peak function.$)$
There exists a function $f\in R$ with $\|f\|=1$,  $f(x)=1$, and  such that for every neighbourhood $U$ of $x$,
we have   $|f(y)|<1$ for all $y\in X\setminus U$.
\item  There exists a constant
  $K$, such that for every neighbourhood $U$ of $x$, and every
  $\epsilon>0$, there exists an $f\in R$ with $\|f\|<K$, $f(x)=1$, and
  $|f(y)|<\epsilon$ for all $y\in X\setminus U$.
\end{enumerate}
\end{proposition}

\noindent In (2), the point $x$ is referred to as a {\em weak peak point}. 

\begin{proposition}[Cohen Factorization Theorem]
Let $R$ be a commutative unital Banach algebra, $ M $ a maximal 
ideal of $R$, and suppose that $ M $ has an approximate identity. Then
for every $f\in M $, there exist $f_1,f_2\in M $ such that $f=f_1f_2$.
\end{proposition}

\noindent An immediate consequence of these results is the following. 

\begin{corollary}
Let $R$ be a uniform algebra on $X$, and let $x\in X$. Suppose that 
$x$ is a weak-peak point. 
Set $M :=\{f\in R:f(x)=0\}$. Then  
$$
ki( M )= M .
$$
\end{corollary}
\begin{proof} 
Since $M$ is maximal, by the Cohen Factorization Theorem, we have 
$M^2= M$.
\end{proof}

 \noindent Thus, for every uniform algebra $R$ we have ``many'' ideals $M$ with $M^2=M$, 
 namely any maximal ideal $M$ of the form $M_x=\{f\in R:f(x)=0\}$, where 
$x$ is a weak peak point. See \cite[p. 101]{B} and also \cite{DU}. 
We emphasize that the set of weak peak points (sometimes called 
the {\em Choquet boundary of $R$}; see also \cite[Definition on p.87 and Theorem 2.3.4]{B}) 
is dense in the \v{S}ilov  boundary of $R$  (by \cite[Corollary~4.3.7(i)]{Dal}). 
We recall here that a closed subset $F$ of $X$ is called a {\em closed 
boundary} for $R$ if 
$$
\sup_{x\in F}|f(x)|=\sup_{x\in X}|f(x)|.
$$
The intersection of all the closed boundaries for $R$ is 
called the {\em \v{S}ilov boundary of $R$}. 

\begin{example}[Disk algebra and Wiener algebra]  
As illustrative examples, consider the disk algebra and the Wiener 
algebra. Let
$$
\mD:=\{z\in \mC:|z|<1\},
$$
and set  
\begin{eqnarray*}
A(\mD)&:=&\{ f\in H(\mD):  
f \textrm{ has a continuous extension to }\mD\cup \partial \mD\},
\phantom{\Big\{}\\
W^+(\mD)&:=&\Big\{f=\sum_{n=0}^\infty a_nz^n\in H(\mD): 
\|f\|_1:=\sum_{n=0}^\infty |a_n|<\infty\Big\},
\end{eqnarray*}
with pointwise operations. $A(\mD)$ is endowed with the sup-norm
$\|\cdot\|_\infty$, while $W^+(\mD)$ is endowed with the
$\|\cdot\|_1$-norm defined above. The maximal ideal
 $
 M:=\{f\in A(\mD):f(1)=0\}
$ 
has an approximate identity given by the sequence $(1-p^n)_{n\in
  \mN}$, where $p$ is the peak function given by
$$
p:=\frac{1+z}{2},\quad z\in \mD,
$$
(for details
of the proof, we refer the reader to \cite[Theorem~6.6]{Sas}.)

Let $(r_n)_{n\in \mN}$ be any sequence such that $r_n \searrow 1$, and
$$
e_n(z):= \frac{z-1}{z-r_n}, \quad n\in \mN.
$$
Then $(e_n)_{n\in \mN}$ is a bounded approximate identity for
$$
 M:=\{f\in W^+(\mD) : f(1)=0\}.
$$
A rather lengthy proof of this result in the case when
$$
r_n=1+\frac{1}{n}, \quad n\in \mN,
$$ 
can be found in \cite{Kou}, while the result is also mentioned without
proof in \cite{Fai}. A short proof due to the first author of this
article can be found in \cite{MorSas} or in \cite[Theorem~6.10]{Sas}.
\hfill$\Diamond$
\end{example}

\section{$ki(I)$ for ideals $I$ in $H^\infty(\mD)$}
\label{section_4}

\noindent Let $H^\infty(\mD)$ denote the algebra of all bounded holomorphic
functions in $\mD$. We sometimes write $\H$ instead of $H^\infty(\mD)$. 
The spectrum (or maximal ideal space), $M(\H)$  of $\H$ is the set of nonzero multiplicative 
linear functionals on $\H$.

\begin{observation}
\label{free}
Let $I$ be an ideal in $\H$. Suppose that $I$ is a non-free ideal; 
that is,
$$
Z_\D(I):=\displaystyle \Inter_{f\in I} Z_{\D}(f)\not=\emp.
$$
Then  $ki(I)=\{0\}$.
\end{observation}
\begin{proof}
If $Z_\D(I)=\D$, then $I=(0)$ and so $ki(I)=\{0\}$. So suppose that there exists an isolated point $z_0\in Z_\D(I)$. 
Let $f\in ki(I)$ and $n\in \N$ be given. Then $\ord(f,z_0)\geq n$. Hence $f\equiv 0$.  Again $ki(I)=\{0\}$.
\end{proof}

\noindent A description of the maximal ideals $M$ in $\H$ with
$ki(M)=\{0\}$ is already implicit in Kenneth Hoffman's work
\cite{hof}.  Recall that $\widehat{f}\in C(M(\H);\mC)$ denotes the 
{\em Gelfand transform} of $f\in \H$
$$
m\stackrel{\widehat{f}}{\longmapsto} m(f)=: \widehat{f}(m), 
\quad m\in M(\H).
$$

 For $m\in M(\H)$, and $f\in \H$, let us define
$$
\ord (f,m)=\ord (\widehat f\circ L_m,0),
$$
where $L_m:\D\to P(m)$ is
the {\em Hoffman map} associated with $m$; that is 
$$
L_m(z)=\lim \frac{z+z_\alpha}{1+\ov z_\alpha z},
$$
where $(z_\alpha)$ is a net in $\D$ converging to $m$ (all limits
being taken in the $\textrm{weak-}\ast$/Gelfand topology of
$M(\H)$). Note that $\widehat f\circ L_m$ is holomorphic in $\D$. In
particular, $\ord(f,m)=\infty$ if and only if $\widehat f\equiv 0$ on
$P(m)$, where $P(m)$ denotes the Gleason part containing $m$. Recall the
pertinent definitions below.

\begin{definition}[Pseudohyperbolic distance and the Gleason part] 
The {\em pseudohyperbolic distance between two points 
$m,\widetilde{m}\in M(\H)$} is defined by 
$$
\rho(m,\widetilde{m})
:=
\sup \Big\{ |\widehat{f}(\widetilde{m})|: 
f\in H^\infty,\;\|f\|_\infty\leq 1, \;\widehat{f}(m)=0\Big\}.
$$
For $m\in M(\H)$, let 
$$
P(m):=\{\widetilde{m}\in M(\H):\rho(m,\widetilde{m})<1\}
$$
denote the {\em Gleason part of $M(\H)$ containing $m$}. 
\end{definition}

\noindent By \cite{hof},
$$
\ord (f,m)=
\sup\Big\{n\in \N:  f=f_1\dots f_n,\;\; \widehat f_j(m)=0 
\textrm{ for all }j=1,\dots, n\Big\}.
$$

\begin{lemma}\label{ord}
Let $I$ be an ideal in $\H$. Then $\ord (f,m)=\infty$ 
for every $f\in ki(I)$ and $m\in Z(I)$.
\end{lemma}

\noindent Here $Z(I):=Z_{M(\H)}(I):=\displaystyle \Inter_{f\in I}
\Big\{m\in M(\H):\widehat{f}(m)=0 \Big\}.$
\begin{proof}
Let $f\in ki(I)$. Fix $n\in\N$. Then 
$$
f=\sum_{k=1}^K f_{k,1}\dots f_{k,n}
$$ 
for $f_{k,\ell}\in I$.  In particular, $\widehat{f_{k,\ell}}(m)=0$ for
every $m\in Z(I)$.  Hence $\ord (f,m)\geq n$. Since $n$ was arbitrary,
we conclude that $\ord (f,m)=\infty$.
\end{proof}

\begin{theorem}
Let $ M$ be a maximal ideal in $\H$ and $m\in M(\H)$ with 
$\ker m= M$. Then the following assertions are equivalent:
\begin{enumerate}
\item[(1)] $ki( M)= M$.
\item[(2)]  $ M$ does not contain an interpolating Blaschke product.
\item[(3)] The Gleason part $P(m)$ containing $m$ is the singleton
  $\{m\}$.
\end{enumerate}
In all cases, that is for all maximal ideals in $\H$,
$$
ki( M)=I\Big(\ov {P(m)},\H\Big)
:=\Big\{f\in \H: \widehat{f}\equiv 0\;\textrm{ on }\ov {P(m)}\Big\},
$$
where $\ov E$ denotes the closure of the set $E\ss M(\H)$.
In particular, $ki( M)$ is a closed prime ideal, and 
$$
ki( M)=\{0\}\ssi  M= M_{z_0}
:=\{f\in\H: f(z_0)=0\}\textrm{ for some }z_0\in \D.
$$
\end{theorem}
\begin{proof} By \cite{hof} (see also \cite{gar}), the statements 
(2) and (3) are equivalent.  If (3) holds, then by \cite{hof}, $ M= M^2$ 
(even in the strict sense: each $f\in M$ can be written as $f=g\cdot
h$, where $g,h\in M$). So $ki( M)= M$.

If $b$ is an interpolating Blaschke product contained in $ M$, then
$b\notin  M^2$, because otherwise 
$$
b=\sum_{k=1}^K f_k g_k,
$$
with $f_k,g_k\in M$.  Hence $\ord (b,m)\geq 2$, a contradiction; see
\cite{hof}.  This shows the equivalence of (1), (2) and (3).

To prove the rest, we note that
$$
\Big\{f\in \H: \widehat{f}\equiv 0\;\textrm{ on }\ov {P(m)}\Big\}
=
\{f\in \H: \ord(f,m)=\infty\}.
$$
Moreover, for every $n\in\N$, any such $f$ admits a factorization of
the form $f=g_1\cdots g_n$ with $\widehat{g_k}(m)=0$.  Hence
$$
I\Big(\ov{P(m)},\H\Big)\ss ki( M).
$$
Conversely, if $f\in ki( M)$, then $f$ is a sum of functions in $ M$
each having order at least $n$ at $m$. Thus $\ord (f,m)=\infty$ for every
$f\in ki( M)$ and so $\widehat{f}\equiv 0$ on $P(m)$; see
\cite{hof}. Thus
$$
ki( M)=I\Big(\ov {P(m)},\H\Big).
$$

\noindent Since 

\smallskip 

$\bullet\;$ $\ov {P(m)}$ is a proper subset of $M(\H)$ if and only if
$m\in M(\H)\setminus\D$, and

$\bullet\;$ $P(z_0)=\D$ for every $z_0\in \D$,

\smallskip 

\noindent we conclude that $ki( M)=\{0\}$ if and only if $ M= M_{z_0}$.

\medskip 

\noindent It is easily seen and well-known that 
$ I\Big(\ov{P(m)},\H\Big) $ is a closed prime ideal.
\end{proof}

\noindent Using Izuchi's \cite{izu} extensions of Hoffman's
factorization theorems, we also obtain the following result:

\begin{proposition}
Let $E\ss M(\H)$ be a hull-kernel closed subset of $M(\H)$, that is, 
$E$ is the zero-set of the ideal
$$
I(E,\H) = \Big\{f\in\H:\widehat{f}|_E\equiv 0\Big\}.
$$
Suppose that $E$ is a union of Gleason parts. Then
$$
ki\big(I(E,\H)\big)=I(E,\H).
$$
\end{proposition}
\begin{proof}
Let $f\in I(E,\H)$ and $n\in \N$. Since we have that $\ord(f,m)=\infty$ 
for every $m\in E$ (because, by hypothesis, $m\in E$ implies 
$P(m)\ss E$), it follows that $f\in I(E,\H)$ has a factorization 
$f=f_1\cdots f_n$, with $f_k\in I(E,\H)$; see \cite{izu}.  
Conversely, if $f\in ki(I(E,\H))$, then $\ord (f,m)=\infty$ 
for every $m\in E$.  This yields the assertion.
\end{proof}

\begin{corollary}
Let $I$ be a non-maximal closed prime ideal in $\H$.Then $ki(I)=I$.
\end{corollary}
\begin{proof} 
By \cite{gomo}, every non-maximal closed prime ideal in $\H$ 
has the form $I=I(E,\H)$, where $E=Z(I)$ is a union of Gleason parts.
\end{proof}

\noindent We will now collect a few technical results, which will be
used in the proof of Proposition~\ref{prop_RM} below.

Let $g\in \H$ be zero-free. Suppose that $\|g\|_\infty\leq 1$. Then
there exists a positive measure $\mu$ on the unit circle $\mT$ such
that
$$
g(z)=g_\mu(z):=\textrm{exp}\left(\int_\T  \frac{z+\xi}{z-\xi}\;d\mu(\xi)\right).
$$
If $\xi=e^{it}$, this $\mu$ has the form 
$$
d\mu(\xi)=\log\frac{1}{|g(e^{it})|} dt +d\mu_s(\xi),
$$
where $\mu_s$ is singular with respect to Lebesgue measure on $\T$.

The following result corresponds to assertion (1.1) in
\cite[p. 170]{mo}, given there without proof.

\begin{lemma}
\label{lemma_rm}
Let $g=g_\mu\in \H$ be zero-free and suppose that 
$\|g\|_\infty\leq 1$. Then, for every $z\in \D$,
$$
|1-g(z)|\leq \frac{1+|z|}{1-|z|}\;\mu(\T).
$$
\end{lemma}
\begin{proof}
First we note that for $w\in \C$ with  ${\rm Re}\; w\leq 0$,
$$|1-e^w|=\left|\int_{[0,w]}e^\zeta\;d\zeta\right|\leq  |w|.$$
Since $\displaystyle {\rm Re}\; \frac{z+\xi}{z-\xi}=
\frac{|z|^2-1}{\;|z-\xi|^2\;}\leq 0 $ and $\mu\geq 0$, 
we deduce that
$$
\phantom{aaaaaaaaaaa}
|1-g(z)|\leq\left|\int_\T  \frac{z+\xi}{z-\xi} \;d\mu(\xi)\right|
\leq
\frac{1+|z|}{1-|z|}\mu(\T).\phantom{aaaaaaaaaaa}\qedhere
$$
\end{proof}

\begin{proposition}
\label{prop_RM}
If $ P$ is a prime ideal in $\H$, then $ki( P)=\{0\}$ if and only if 
$P$ is a maximal ideal of the form $ M_{z_0}$ for some $z_0\in \D$.
\end{proposition}
\begin{proof}
We have already seen that $ki( M_{z_0})=\{0\}$. If 
$Z_{\D}( P)\inter \D\not=\emp$, then it easily follows that $ P\ss
(z-z_0)\H$ for all $z_0 \in Z_\D(P) \cap \D$.  Due to primeness
$z-z_0\in P$ (each $f \in P$ factors as $f=(z-z_0)^n g$, where $n$ is
the order of the zero $z_0$, but then $g \not\in P$, so $z-z_0 \in
P$), and so $ P= M_{z_0}$ again. Now suppose that $Z_\D( P)=\emp$;
that is, $ P$ is a free prime ideal.  We show that $ki( P)$ contains
elements different from the zero function.

\medskip 
  
\noindent {\bf Case $1^\circ$} Suppose that $ P$ contains a Blaschke
product $B$, with zero sequence $(z_n)$ (multiplicities included). In
particular,
$$
\sum_{n=1}^\infty (1-|z_n|^2)<\infty.
$$
For each $k$, choose a tail of the sequence so that
$$
\sum_{n=N_k}^\infty (1-|z_n|^2) \leq \frac{1}{2^{k}}.
$$
Let $B_k$ be the Blaschke product associated with these zero
sequences. Since $B_k$ differs from $B$ only by finitely many zeros,
the freeness of $ P$ implies that $B_k\in P$ (otherwise we would have
$z-z_0 \in P$, hence $P=M_{z_0}$ again). 
Since
$$
\sum_{k=1}^\infty \sum_{j=N_k}^\infty (1-|z_j|^2)<\infty,
$$
the collection of all zeros of all $B_k$ is a Blaschke sequence again.
Hence, due to absolute convergence of the associated products, any
reordering converges again, and so
$$
B_*:=\prod_{k=1}^\infty B_k
$$ 
is a Blaschke product again.  Clearly, $B_*\in ki( P)$.

\medskip 

\noindent {\bf Case $2^\circ$} Let $Bg\in P$, where $g$ is a zero-free
function, and we may assume that $\|g\|_\infty \leq 1$. Either $B\in
P$ (and we are done by the first case) or $g\in P$.  Since $g$ has
roots of any order, we see that $g^{1/n}\in P$ for every $n$. Choose
$n_k$ going to infinity so fast that
$$
\sup_{|z|\leq 1-1/k}|1-g^{1/n_k}(z)|  <\frac{1}{2^{k}} , 
$$
which is possible by Lemma~\ref{lemma_rm} above. Then the infinite
product
$$
h=\prod_{k=1}^\infty  g^{1/n_k}
$$
converges locally uniformly to a function $h\in \H$. Clearly, 
$h\in ki( P)$.
\end{proof}

\goodbreak 

\begin{remark}
The proof above shows the following:
\begin{enumerate}
\item If $I$ is any free ideal in $\H$ containing a Blaschke product,
  then $ki(I)\not=\{0\}$.
\item If $I$ is any free ideal in $\H$ containing a zero-free function
  $g$ and all of its roots, then $ki(I)\not=\{0\}$.
\end{enumerate}
\end{remark}

\noindent Let us also remark that there do exist free ideals with
$ki(I)=\{0\}$, as demonstrated below.

\begin{observation}
$ki( S\H)=\{0\}$, where $S$ is the atomic inner function 
$$
S(z)=\textrm{exp}\Big(-\frac{1+z}{1-z}\Big),\quad z\in \D.
$$ 
\end{observation}
\begin{proof}
If $f\in ki(S\H)$, $f\not\equiv 0$, then, for every $n$,
$$
f=\sum_{k=1}^m \prod_{j=1}^n (h_{kj} S)= h_nS^n .
$$
In particular, $S^n$ divides the inner factor $\varphi$ of $f$ for
every $n$, say $\varphi=u_nS^n$ for inner functions $u_n$.  This is
impossible though, because $S^n$ goes to zero locally uniformly in
$\D$, and so due to the boundedness of $u_n$, $\varphi =0$.
\end{proof}

\begin{observation}
Let $I$ be  a countably generated  free prime ideal in $\H$. Then $ki(I)=I$.
\end{observation}
\begin{proof}
By \cite{gomo05,mo}, $I$ is generated by $\{S_\alpha(z)^{1/n}: n\in \N\}$,
where 
$$
S_\alpha(z)=\textrm{exp}\left(-\frac{\alpha+z}{\alpha-z}\right)
$$
for some $\alpha\in T$. But $I=\{h\; S_\alpha ^{1/n}: n\in\N, h\in \H\}$.
Hence, given $n\in\N$, every $f=hS_\alpha^{1/p}\in I$ can be factorized as
$$
f=h \underbrace{S_\alpha^{1/(pn)}\dots S_\alpha^{1/(pn)}}_{n-times}.
$$
So $f\in I^n.$
\end{proof}

\end{document}